\documentclass[reqno]{amsart}
\usepackage{amssymb,latexsym}
\usepackage{amsfonts,mathrsfs}
\usepackage[margin=1.4in]{geometry}
\usepackage{amsmath}
\usepackage{enumerate}
\usepackage{mathtools}
\usepackage{tikz}
\usetikzlibrary{arrows.meta}

\newtheorem{thm}{Theorem}[section]
\newtheorem{lemma}[thm]{Lemma}

\newtheorem{definition}[thm]{Definition}

\theoremstyle{remark}
\newtheorem{remark}{Remark}[section]

\newcommand{\Vol}{\mathrm{Vol}}

\begin{document}
\bibliographystyle{abbrv}

\title{Scalar curvature lower bounds on asymptotically flat manifolds}
\author{Yuqiao Li}
\address{Department of Mathematics, Hefei University of Technology, Hefei, 230009, P.R.China.} 
\email{lyq112@mail.ustc.edu.cn}
\thanks{This paper is supported by Initial Scientific Research Fund of Young Teachers in Hefei university of Technology of No.JZ2024HGQA0122}

\begin{abstract}
In this paper, we consider the scalar curvature in the distributional sense of \cite{MR3366052}  and the scalar curvature lower bound in the $\beta-$weak $(\beta\in(0, \frac{1}{2}))$ sense of \cite{MR4685089} on an asymptotically flat $n-$manifold with a $W^{1,p}(p>n)$ metric.
We first show that the scalar curvature lower bound under the Ricci-DeTurck flow depends on the scalar curvature lower bound in the $\beta-$weak sense and the time.
Then we prove that the lower bound of the distributional scalar curvature of a $W^{1, p}$ metric coincides with the lower bound of the scalar curvature in the $\beta-$weak sense at infinity.
\end{abstract}

\maketitle

\numberwithin{equation}{section}

\section{Introduction}

Defining the scalar curvature in weak sense on Riemannian manifolds with non-smooth metrics is of great meaning in many geometric problems.
Lee-LeFloch have defined the distributional scalar curvature of $W^{1,p}(p\geq n)$ metrics on asymptotically flat spin manifolds to study the positive mass theorem in low regularity \cite{MR3366052}.
For $C^0$ metrics, Gromov introduced a definition of the scalar curvature lower bound \cite{MR3201312}.
Since the scalar curvature lower bound is preserved by the Ricci flow of a smooth initial metric,
Burkhardt-Guim has defined the pointwise lower scalar curvature bounds for $C^0$ metrics on compact manifolds by Ricci flow \cite{MR4034918}.
For these two definitions of the scalar curvature lower bounds, Jiang-Sheng-Zhang proved that the distributional scalar curvature lower bound on a compact manifold is preserved along the Ricci flow for $W^{1,p}(p>n)$ initial metrics \cite{MR4596043}, which implies that the scalar curvature has lower bound in the sense of Burkhardt-Guim \cite{MR4034918}.

Recently, Burkhardt-Guim also defined the scalar curvature lower bound in the $\beta-$weak sense and the ADM mass on an asymptotically flat manifold with a $C^0$ metric \cite{MR4685089}.
This is an extension of the definition of the scalar curvature lower bound introduced in \cite{MR4034918} to noncompact manifolds.

In this paper, we consider the scalar curvature lower bound in the distributional sense of \cite{MR3366052} and in the $\beta-$weak sense of \cite{MR4685089} on an asymptotically flat manifold with a $W^{1,p}(p>n)$ metric.
In another word, we consider the result of Jiang-Sheng-Zhang in \cite{MR4596043} on asymptotically flat manifolds.

Firstly, we show the scalar curvature lower bound in the $\beta-$weak sense is almost preserved under the Ricci flow by the approach of Bamler \cite{MR3512888}.

\begin{thm}\label{Rt01}
  Let $M^n$ be a smooth manifold with a $C^0-$asymptotically flat metric $g$ satisfying \eqref{AF}.
 For some $\beta\in(0, \frac{1}{2})$, suppose $g$ has scalar curvature bounded below by $\kappa\in\mathbb{R}$ in the sense of Ricci flow on $M\backslash K$, where $K$ is a compact set (See Definition \ref{scalar} in Section 2).
  Suppose $g_0$ is a $C^0$ metric on $\mathbb{R}^n$ constructed in Section 2 that agrees with $\Phi_*g$ on $B_{\delta}(x, \frac{1}{2^{\beta}-1}t^{\beta})$, for $x\in\mathbb{R}^n\backslash\overline{B(0, 1)}$, $t>0$, where $\Phi$ is a coordinate chart defined as in Definition \ref{scalar}.
  Let $(g_t)_{t>0}$ be the Ricci-DeTurck flow for $g_0$ on $\mathbb{R}^n$ with \eqref{gt}\eqref{dgt}.
  Then we have
  \begin{equation}\label{lb}
    R(x, t)\geq\kappa-Ct^{\lambda},
  \end{equation}
  where $C=C(n)$ is a constant and $\lambda>0$.
\end{thm}

 From the above theorem, we see that the scalar curvature along the Ricci-DeTurck flow at positive time at a point is bounded below by the scalar curvature lower bound in the sense of Ricci flow on a neighborhood of the point.
 
 Then we prove that for a $W^{1, p} (p>n)$ metric on an asymptotically flat $n-$manifold, the distributional scalar curvature lower bound coincides with the scalar curvature lower bound in the sense of Ricci flow at infinity.
 Since the distributional scalar curvature is not pointwise, we can only show the theorem at infinity.

 \begin{thm}\label{dfr}
    Let $M^n$ be an asymptotically flat manifold with a metric $g\in C^0\bigcap W^{1, p}_{-\tau}$, $p>n$, $\tau>\frac{n-2}{2}$. If the scalar curvature distribution $\ll R_g, u\gg\geq\kappa$, for any compactly supported nonnegative smooth function $u$ and some constant $\kappa\geq0$, then for some $\beta\in(0, \frac{1}{2})$, $g$ has scalar curvature bounded below by $\kappa$ in the sense of Ricci flow at infinity.
\end{thm}

This paper is organized as follows. In Section 2, we introduce necessary definitions of the scalar curvaure lower bound. We prove Theorem \ref{Rt01} and Theorem \ref{dfr} in Section 3 and 4 respectively.

\section{Pointwise scalar curvature lower bound of $C^0$ metrics}

In \cite{MR4034918}, Burkhardt-Guim defined the pointwise scalar curvature lower bounds on compact manifolds with $C^0$ metrics via regularizing Ricci flow (see Theorem 1.1 of \cite{MR4034918}) which is preserved under the Ricci flow.
The idea comes from the property that the scalar curvature lower bounds of smooth metrics are preserved by the Ricci flow.

\begin{definition}\label{cpt}(Definition 1.2 of \cite{MR4034918})
  Let $M$ be a closed manifold and $g_0$ be a $C^0$ metric on $M$.
  For some $\beta\in(0, \frac{1}{2})$, we say $g_0$ has scalar curvature bounded below by $\kappa$ at $x$ in the $\beta$-weak sense if there exists a regularizing Ricci flow $(g(t), \chi)$  with $g_0$ as the initial data, where $\chi: M\rightarrow M$ is a surjection, such that for any $y\in M$ with $\chi(y)=x$, we have
  \begin{equation}\label{C0scalar}
    \inf_{C>0}\left(\liminf_{t\rightarrow0}\left(\inf_{B_{g(t)}(y, Ct^{\beta})}R(g(t))\right)\right)\geq\kappa,
  \end{equation}
  where $R(g(t))$ is the scalar curvature of the regularizing Ricci flow $g(t)$ which is smooth for positive time.
\end{definition}

For simplicity, we use the notation $R_{C^0_{\beta}}(g_0)(x)\geq\kappa$ to mean $g_0$ has scalar curvature bounded below by $\kappa$ at $x$ in the $\beta$-weak sense satisfying \eqref{C0scalar}.
Burkhardt-Guim also showed that the scalar curvature lower bound in the $\beta$-weak sense is preserved under the Ricci flow(Theorem 1.5 in \cite{MR4034918}).
Except for the definition of the scalar curvature lower bound in some weak sense for $C^0$ metrics by Gromov(\cite{MR3201312}), Burkhardt-Guim gave three more definitions of the scalar curvature lower bound, which are proven to be equivalent.

In order to investigate the mass of an asymptotically flat manifold with a $C^0$-asymptotically flat metric, Burkhardt-Guim generalized the definition of the scalar curvature lower bound in the $\beta$-weak sense (Definition \ref{cpt}) to noncompact manifolds in the following way (\cite{MR4685089}).

\begin{definition}\label{scalar}(Definition 2.3 of \cite{MR4685089})
  Let $M^n$ be a smooth manifold and $g$ be a continuous metric on $M$.
  For some $\beta\in(0, \frac{1}{2})$, we say that $g$ has scalar curvature bounded below by $\kappa\in\mathbb{R}$ at $x\in M$ in the $\beta$-weak sense if there exists a coordinate chart
  $\Phi: \mathcal{U}_x\rightarrow\Phi(\mathcal{U}_x)$ on $M$, where $\mathcal{U}_x\subset M$ is a neighborhood of $x$ and $\Phi(\mathcal{U}_x)\subset\mathbb{R}^n$ is an open subset of $\mathbb{R}^n$, and there exists a continuous metric $g_0$ on $\mathbb{R}^n$ and a Ricci-DeTurck flow $(g_t)_{t\in(0, T]}$ for $g_0$ with respect to the Euclidean background metric, such that
  \[ g_0\big|_{\Phi(\mathcal{U}_x)}=\Phi_*g, \]
  \begin{equation}\label{C0scalar1}
    \inf_{C>0}\left(\liminf_{t\rightarrow0}\left(\inf_{B_{\delta}(\Phi(x), Ct^{\beta})}R(g_t)\right)\right)\geq\kappa.
  \end{equation}
\end{definition}

We also denote the above definition of the scalar curvature lower bound by $R_{C^0_{\beta}}(g)(x)\geq\kappa$.
We say that $g$ has scalar curvature bounded below by $\kappa\in\mathbb{R}$ in the sense of Ricci flow on an open region $U\subset M$ if for some $\beta\in(0, \frac{1}{2})$, all $x\in U$, $g$ has scalar curvature bounded below by $\kappa$ in the $\beta-$weak sense at $x$.
The idea of extending the definition of the scalar curvature lower bound in the $\beta-$weak sense on closed manifolds (Definition \ref{cpt}) to an asymptotically flat manifold is based on the existence of the Ricci-DeTurck flow starting from a $L^{\infty}$ metric close to the standard Euclidean metric $\delta$, which is given by Koch-Lamm in \cite{MR2916362}.

\begin{definition}
 The Ricci-DeTurck flow on $\mathbb{R}^n$ for a metric $g_0$ with respect to the Euclidean background metric is a family of metrics $g(t)$ satisfying
\begin{equation}
\begin{aligned}\label{RD}
  \partial_tg(t)&=-2Ric(g(t))-\mathcal{L}_{X_{\delta}(g(t))}g(t) \text{\ \ in\ \ } \mathbb{R}^n\times(0, T) \\
  g(\cdot, 0)&=g_0,
\end{aligned}
\end{equation}
where
\[ X_{\delta}(g)=-g^{ij}\Gamma^k_{ij}e_k \]
is a vector field and $\{e_k\}_{k=1}^n$ is an orthonormal frame of $g(t)$, $\Gamma$ is the Christoffel symbol of $g(t)$.
\end{definition}

\begin{thm}\label{RDF}(Theorem 4.3 in \cite{MR2916362})
  Suppose $g_0\in L^{\infty}(\mathbb{R}^n)$ is a metric on $\mathbb{R}^n$ satisfying $||g_0-\delta||_{L^{\infty}(\mathbb{R}^n)}<\epsilon$, for some $\epsilon<1$.
  Then there exists a constant $c=c(n)>0$ and a global analytic solution $g(t)_{t>0}$ to the Ricci-DeTurck flow \eqref{RD} such that for $t>0$,
  \[ \lim_{t\rightarrow0} g(t)=g_0, \]
  \begin{equation}\label{gt}
    ||g(t)-\delta||_X\leq c(n)||g_0-\delta||_{L^{\infty}(\mathbb{R}^n)},
  \end{equation}
  where the norm $||\cdot||_X$ is defined by
  \begin{equation}\label{hx}
  \begin{split}
    ||h||_X=&\sup_{t>0}||h(t)||_{L^{\infty}(\mathbb{R}^n)}\\
    &+\sup_{x\in\mathbb{R}^n}\sup_{r>0}
    \left( r^{-n/2}||\nabla h||_{L^2(B(x,r)\times(0, r^2))}+r^{\frac{2}{n+4}}||\nabla h||_{L^{n+4}(B(x, r)\times(\frac{r^2}{2}, r^2))} \right),
  \end{split}\end{equation}
  and, for all $k\in\mathbb{N}$, there exists $c_k(n)>0$ such that for all $t>0$,
  \begin{equation}\label{dgt}
    ||\nabla^k(g(t)-\delta)||_{L^{\infty}(\mathbb{R}^n)}\leq c_k(n)||g_0-\delta||_{L^{\infty}(\mathbb{R}^n)}t^{-\frac{k}{2}}.
  \end{equation}
\end{thm}

Now we consider the scalar curvature lower bound in the $\beta-$weak sense of an $n$-dimensional smooth manifold $M$ with a $C^0-$asymptotically flat metric $g$.
Namely, for some compact set $K\subset M$, there exists a diffeomorphism $\Phi: M\backslash K\rightarrow\mathbb{R}^n\backslash\overline{B(0,1)}$, such that,
for $|x|_{\delta}\geq r_0$, $x\in\mathbb{R}^n\backslash\overline{B(0,1)}$, $r_0>1$, we have
\begin{equation}\label{AF}
  |(\Phi_*g)_{ij}-\delta_{ij}|(x)=O(|x|_{\delta}^{-\tau}),
\end{equation}
where $\tau>\frac{n-2}{2}$.
For any $x\in M\backslash K$, there is a  neighborhood $\mathcal{U}_x\subset M\backslash K$ of $x$ and a diffeomorphism $\Phi: \mathcal{U}_x\rightarrow\Phi(\mathcal{U}_x)$, where $\Phi(\mathcal{U}_x)\subset\mathbb{R}^n\backslash\overline{B(0, 1)}$.
We can construct a metric $g_0\in C^0(\mathbb{R}^n)$ which is a continuous extension of $\Phi_*g$ on $\Phi(\mathcal{U}_x)$ to $\mathbb{R}^n$ by gluing the Euclidean metric $\delta$ on the complement of a neighborhood of $\Phi(\mathcal{U}_x)$.
In fact, there is an open set $V\subset\subset\mathbb{R}^n\backslash\overline{B(0, 1)}$ such that $\Phi(\mathcal{U}_x)\subset\subset V$ and a smooth cutoff function $\chi: \mathbb{R}^n\rightarrow [0, 1]$ with $\chi\equiv1$ on $\Phi(\mathcal{U}_x)$ and $Supp(\chi)\subset V$.
Let $g_0 = \chi\Phi_*g+(1-\chi)\delta$ be a continuous metric on $\mathbb{R}^n$.
Then
\[ g_0\big|_{\Phi(\mathcal{U}_x)}=\Phi_*g. \]
And we have
 \[ ||g_0-\delta||_{C^0(\mathbb{R}^n)}=||\chi(\Phi_*g-\delta)||_{C^0(V)}<\epsilon \]
 provided that $||\Phi_*g-\delta||_{C^0(\mathbb{R}^n)}<\epsilon$.
From \eqref{AF}, for $|x|_{\delta}\geq r_0$, we know $||\Phi_*g-\delta||_{C^0(\mathbb{R}^n\backslash\overline{B(0, r_0)})}<cr_0^{-\tau}=\epsilon<1$.
Thus, by Theorem \ref{RDF}, there exists a global smooth solution $g(t)$ to the Ricci-DeTurck flow for $g_0$ satisfying \eqref{gt}\eqref{dgt}.
Therefore, Definition \ref{scalar} can be used to define the scalar curvature lower bound at any $x\in M\backslash K$ and Definition \ref{cpt} can be used to define the scalar curvature lower bound on compact set $K$.
Therefore, the scalar curvature lower bound in the $\beta$-weak sense can be defined on an asymptotically flat manifold $M$ with $C^0$ metric $g$.

\section{Scalar curvature lower bound under the Ricci-DeTurck flow}
From Definition \ref{scalar}, the scalar curvature lower bound in the $\beta-$weak sense of a $C^0$ metric on an asymptotically flat manifold depends on the scalar curvature lower bound of the Ricci-DeTurck flow $g(t)$ for $g_0$ on $\mathbb{R}^n$.
Conversely, we would like to investigate the scalar curvature lower bound of the Ricci-DeTurck flow $g(t)_{t>0}$ at positive time from the scalar curvature lower bound of $g_0$ in the $\beta-$weak sense.

\begin{thm}\label{Rt0}
  Let $M^n$ be a smooth manifold with a $C^0-$asymptotically flat metric $g$ satisfying \eqref{AF}.
 For some $\beta\in(0, \frac{1}{2})$, suppose $g$ has scalar curvature bounded below by $\kappa\in\mathbb{R}$ in the sense of Ricci flow on $M\backslash K$, where $K$ is a compact set.
  Suppose $g_0$ is a $C^0$ metric on $\mathbb{R}^n$ constructed in Section 2 that agrees with $\Phi_*g$ on $B_{\delta}(x, \frac{1}{2^{\beta}-1}t^{\beta})$, for $x\in\mathbb{R}^n\backslash\overline{B(0, 1)}$, $t>0$, where $\Phi$ is a coordinate chart defined as in Definition \ref{scalar}.
  Let $(g_t)_{t>0}$ be the Ricci-DeTurck flow for $g_0$ on $\mathbb{R}^n$ with \eqref{gt}\eqref{dgt}.
  Then we have
  \begin{equation}\label{lb}
    R(x, t)\geq\kappa-Ct^{\lambda},
  \end{equation}
  where $C=C(n)$ is a constant and $\lambda>0$.
\end{thm}

\begin{proof}
By the assumptions, $g_0$ is a continuous metric on $\mathbb{R}^n$ satisfying $||g_0-\delta||_{C^0(\mathbb{R}^n)}<\epsilon$ and $(g_0-\delta)$ is compactly supported.
From Theorem \ref{RDF}, there exists a global smooth solution $g(t)$ to the Ricci-DeTurck flow for $g_0$ satisfying \eqref{gt}\eqref{dgt}.
By \eqref{dgt}, we know that, for all $(x,t)\in\mathbb{R}^n\times\mathbb{R}^+$,
\begin{equation}\label{lower}
  |R(g(t))|\leq\frac{C_1(n)}{t}.
\end{equation}
The Ricci-DeTurck flow differs from the Ricci flow by a family of diffeomorphisms.
Then the scalar curvature $R$ of $g(t)$ along the Ricci-DeTurck flow \eqref{RD} with respect to the Euclidean background metric evolves by
\begin{equation}\label{sc}
  \partial_t R=\Delta R-\partial_{X_{\delta}(g)}R+2|Ric|^2.
\end{equation}
Denote $X_{\delta}(g)$ by $X$ and $g(t)$ by $g$.
We follow the approach of Lemma 4 in \cite{MR3512888} on the Ricci-DeTurck flow with continuous initial metric.

Consider the scalar heat kernel $\Phi^{RD}(x,t;y,s)$ for the operator $\partial_t-\Delta^{g(t)}+\nabla^{g(t)}_X$.
That is, for fixed $y\in\mathbb{R}^n$ and $0<s<t$,
\[ \partial_t\Phi^{RD}(x,t;y,s)=\Delta^{g(t)}_x\Phi^{RD}(x,t;y,s)-\nabla^{g(t)}_{X,x}\Phi^{RD}(x,t;y,s), \]
and
  \begin{equation}\label{int}
    \int_{\mathbb{R}^n}\Phi^{RD}(x,t;y,s)dg_s(y)=1.
  \end{equation}
From Lemma 3 of \cite{MR3512888}, for any $r>0$, $0<s<t$, we have
  \begin{equation}\label{kernel}
    \int_{\mathbb{R}^n\backslash B(x,r)}\Phi^{RD}(x,t;y,s)dg_s(y)\leq C_2(n)\exp\left(-\frac{r^2}{D(t-s)}\right),
  \end{equation}
  where $B(x,r)$ is the ball in $\mathbb{R}^n$ with respect to the Euclidean metric $\delta$ and $D>4$ is a constant.

 By \eqref{sc}, for any $x\in\mathbb{R}^n$, $0<s<t$, we have
 \begin{equation}\label{scalar1}
     R(x, t)\geq\int_{\mathbb{R}^n}\Phi^{RD}(x,t;y,s)R(y,s)dg_s(y)).
   \end{equation}
We now estimate the lower bound of $R(x,t)$ for $x\in\mathbb{R}^n$ and $t>0$ by the scalar curvature lower bound in the $\beta$-weak sense of $g_0$.
From \eqref{scalar1}, \eqref{kernel}, \eqref{int}, \eqref{lower}, we have
\begin{equation}\label{est}
  \begin{split}
     R(x,t) & \geq \int_{\mathbb{R}^n}\Phi^{RD}(x,t;y,\frac{t}{2})R(y,\frac{t}{2})dg_{\frac{t}{2}}(y) \\
       & = \int_{B(x,r_1)}\Phi^{RD}(x,t;y,\frac{t}{2})R(y,\frac{t}{2})dg_{\frac{t}{2}}(y)
       +\int_{\mathbb{R}^n\backslash B(x,r_1)}\Phi^{RD}(x,t;y,\frac{t}{2})R(y,\frac{t}{2})dg_{\frac{t}{2}}(y)\\
       & \geq a_1\left(1-\int_{\mathbb{R}^n\backslash B(x,r_1)}\Phi^{RD}(x,t;y,\frac{t}{2})dg_{\frac{t}{2}}(y)\right)
       -\frac{C_1}{t/2}\int_{\mathbb{R}^n\backslash B(x,r_1)}\Phi^{RD}(x,t;y,\frac{t}{2})dg_{\frac{t}{2}}(y)\\
       & = a_1-\left(a_1+\frac{C_1}{t/2}\right)\int_{\mathbb{R}^n\backslash B(x,r_1)}\Phi^{RD}(x,t;y,\frac{t}{2})dg_{\frac{t}{2}}(y)\\
       & \geq a_1-\frac{2C_1C_2}{t/2}\exp\left(-\frac{r_1^2}{D(t/2)}\right),
  \end{split}
\end{equation}
where $r_1>0$ is to be chosen later and
\[ a_1=\inf_{B(x, r_1)}R(\cdot, \frac{t}{2})\leq \frac{C_1}{t/2}.\]
Then we can estimate $a_1$ in the same way to get that there is a point $x_1\in B(x, r_1)$, such that
\begin{equation*}
  \begin{split}
    a_1 & = \inf_{B(x, r_1)}R(\cdot, \frac{t}{2})=R(x_1, \frac{t}{2})\\
        & \geq \int_{\mathbb{R}^n}\Phi^{RD}(x_1, \frac{t}{2}; y, \frac{t}{2^2})R(y, \frac{t}{2^2})dg_{\frac{t}{2^2}}(y)\\
        & = \int_{B(x_1, r_2)}\Phi^{RD}(x_1, \frac{t}{2}; y, \frac{t}{2^2})R(y, \frac{t}{2^2})dg_{\frac{t}{2^2}}(y)
        +\int_{\mathbb{R}^n\backslash B(x_1, r_2)}\Phi^{RD}(x_1, \frac{t}{2}; y, \frac{t}{2^2})R(y, \frac{t}{2^2})dg_{\frac{t}{2^2}}(y)\\
        & \geq a_2\left(1-\int_{\mathbb{R}^n\backslash B(x_1, r_2)}\Phi^{RD}(x_1, \frac{t}{2}; y, \frac{t}{2^2})dg_{\frac{t}{2^2}}(y)\right)
        -\frac{C_1}{t/2^2}\int_{\mathbb{R}^n\backslash B(x_1, r_2)}\Phi^{RD}(x_1, \frac{t}{2}; y, \frac{t}{2^2})dg_{\frac{t}{2^2}}(y)\\
        & \geq a_2-\frac{2C_1C_2}{t/2^2}\exp\left(-\frac{r_2^2}{D(t/2^2)}\right),
  \end{split}
\end{equation*}
where $0<r_2<r_1$ is to be chosen later and
\[ a_2=\inf_{B(x_1, r_2)}R(\cdot, \frac{t}{2^2})\leq \frac{C_1}{t/2^2}. \]
Similarly, we can do the steps iteratively.
Then, we obtain a series of radius $r_k$ with $0<r_{k+1}<r_k$, points $x_k\in B(x_{k-1},r_k)$ and numbers
\[ a_k=\inf_{B(x_{k-1}, r_k)}R(\cdot, \frac{t}{2^k})= R(x_k, \frac{t}{2^k}), \]
for $k\in \mathbb{N}^+$ with $x_0=x$, satisfying
\begin{equation}\label{ak}
  a_k\geq a_{k+1}-\frac{2C_1C_2}{t/2^{k+1}}\exp\left(-\frac{r_{k+1}^2}{D(t/2^{k+1})}\right).
\end{equation}
Then \eqref{est} and \eqref{ak} imply that
\begin{equation}\label{final}
  R(x,t)\geq a_k-2C_1C_2\sum_{i=1}^{k}(\frac{2^i}{t})\exp\left(-\frac{r_i^2}{D(t/2^i)}\right).
\end{equation}
By construction, we know $x_k\in B(x, \sum_{i=1}^{k}r_i)$.
Letting $k\rightarrow\infty$, we hope both $\sum_{i=1}^{+\infty}r_i$ and $\sum_{i=1}^{+\infty}(\frac{2^i}{t})\exp\left(-\frac{r_i^2}{D(t/2^i)}\right)$ converge.
By choosing $r_i=(\frac{t}{2^i})^{\beta}$, then $\sum_{i=1}^{+\infty}r_i$ converges provided that $\beta>0$.
And for any $\gamma>1$, we have that
\begin{equation*}
  \begin{split}
     \sum_{i=1}^{+\infty}(\frac{2^i}{t})\exp\left(-\frac{r_i^2}{D(t/2^i)}\right) & =  \sum_{i=1}^{+\infty}(\frac{2^i}{t})\exp\left(-\frac{(t/2^i)^{2\beta-1}}{D}\right)\\
     & \leq\sum_{i=1}^{+\infty}(\frac{2^i}{t})D^{\gamma}(\frac{t}{2^i})^{-(2\beta-1)\gamma}\\
     & =D^{\gamma}t^{-1-(2\beta-1)\gamma}\sum_{i=1}^{\infty}(2^{1+(2\beta-1)\gamma})^i,
  \end{split}
\end{equation*}
which converges when $1+(2\beta-1)\gamma<0$. This is satisfied if $\beta\in(0, \frac{1}{2})$, for any $\gamma>\frac{1}{1-2\beta}>1$.
Consequently, for some $\beta\in(0, \frac{1}{2})$, $r_i=(\frac{t}{2^i})^{\beta}$, there holds
\[ \sum_{i=1}^{+\infty}(\frac{2^i}{t})\exp\left(-\frac{r_i^2}{D(t/2^i}\right)\leq C_3(n, \beta)t^{\lambda}, \]
where $\lambda =-1-(2\beta-1)\gamma>0$. And
\begin{equation*}
  \begin{split}
     \lim_{k\rightarrow\infty}a_k &= \lim_{k\rightarrow\infty}\inf_{B(x_{k-1},(\frac{t}{2^k})^{\beta})}R(\cdot, \frac{t}{2^k})  \\
       & \geq \lim_{k\rightarrow\infty}\inf_{B(x, \frac{1}{2^{\beta}-1}t^{\beta})}R(\cdot, \frac{t}{2^k})\\
       & \geq \liminf_{t\rightarrow0}\inf_{B(x, \frac{1}{2^{\beta}-1}t^{\beta})}R(\cdot, t)\\
       & \geq \inf_{C>0}\liminf_{t\rightarrow0}\inf_{B(x, Ct^{\beta})}R(\cdot, t)\\
       & = R_{C^0_{\beta}}(g_0)(x)
  \end{split}
\end{equation*}
Therefore, by Definition \ref{scalar} and \eqref{final}, for some $\beta\in (0, \frac{1}{2})$, $\lambda>0$, we arrive at,
\[ R(x,t)\geq R_{C^0_{\beta}}(g_0)(x)-2C_1C_2C_3t^{\lambda}\geq\kappa-Ct^{\lambda}. \]

\end{proof}

This concludes Theorem 1.5 of \cite{MR4034918} in the asymptotically flat case which says the preservation of the scalar curvature lower bound in the $\beta$-weak sense of a $C^0$ metric under the Ricci-DeTurck flow.

\section{Scalar curvatures of $W^{1,p}$ metrics in distributional sense}
In this section, we consider an asymptotically flat manifold $M$ with a $W^{1,p}$ metric $g$.
Lee-LeFloch defined the scalar curvature for $W^{1,p}$ metrics in the distributional sense and proved the positive mass theorem for spin manifolds \cite{MR3366052}.
\begin{definition}(Definition 2.1 in \cite{MR3366052})\label{scalar2}
Let $M$ be a smooth Riemannian manifold endowed with a smooth background metric $h$.
Given any Riemannian metric $g\in L^{\infty}_{loc}\bigcap W^{1,2}_{loc}$ on $M$ with
$g^{-1}\in L^{\infty}_{loc}$, for any compactly supported smooth function $u:M\rightarrow \mathbb{R} $, the scalar curvature distribution $R_g$ is defined by
\[ \ll R_g,u \gg :=\int_{M}\left(-V\cdot \bar{\nabla}\left(u\frac{\text{d}\mu_g}{\text{d}\mu_h}\right)+Fu\frac{\text{d}\mu_g}{\text{d}\mu_h}\right)\text{d}\mu_h, \]
where
\[ \Gamma_{ij}^k=\frac12g^{kl}(\bar{\nabla}_ig_{jl}+\bar{\nabla}_jg_{il}-\bar{\nabla}_lg_{ij})=\Gamma_{ij}^k(g)-\bar{\Gamma}_{ij}^k(h), \]
\[ V^k=g^{ij}\Gamma_{ij}^k-g^{ik}\Gamma_{ji}^j=g^{ij}g^{kl}(\bar{\nabla}_j g_{il}-\bar{\nabla}_l g_{ij}), \]
\[ F=g^{ij}\bar{Ric_{ij}}-\bar{\nabla}_k g^{ij}\Gamma_{ij}^k+\bar{\nabla}_kg^{ik}\Gamma_{ji}^j+g^{ij}(\Gamma_{kl}^k\Gamma_{ij}^l-\Gamma_{jl}^k\Gamma_{ik}^l), \]
and $\bar{\nabla}$,$\bar{Ric}$ are the Levi-Civita connection and the Ricci curvature with respect to $h$, $\Gamma_{ij}^k(g)$ and $\bar{\Gamma}_{ij}^k(h)$ denote the Christoffel symbols of $g$ and $h$ respectively, the dot product is taken using the metric $h$, and $\text{d}\mu_h$ and $\text{d}\mu_g$ denote the volume measures of $h$ and $g$, respectively.
\end{definition}

We say the scalar curvature distribution of $g$ satisfies $\ll R_g, u\gg\geq\kappa$ if $\ll R_g, u\gg-\kappa\int_M ud\mu_g\geq0$ for any compactly supported nonnegative function $u$.
Jiang-Sheng-Zhang \cite{MR4480213} generalized the positive mass theorem of Lee-LeFloch from spin manifolds to non-spin manifolds with the metric smooth away from some closed subset.
The author also gave the positive mass theorem of non-spin manifolds with $C^0\bigcap W^{1,p}$ metrics \cite{MR4100922}.
In \cite{MR4480213}, they show the distributional scalar curvature $\ll R_g, u\gg$ is nonnegative assuming the metric is smooth and has nonnegative scalar curvature away from a bounded subset whose measure is limited.
They then demonstrate that the scalar curvature lower bound in the distributional sense is preserved along the Ricci flow on a compact manifold with a $W^{1, p}(p>n)$ metric in \cite{MR4596043}.
As a corollary, the distributional scalar curvature also has lower bound in the $\beta$-weak sense of Definition \ref{cpt}.

Now we use Definition \ref{scalar} to consider the scalar curvature lower bound of an asymptotically flat manifold $M$ with a $W^{1, p}(p>n)$ metric.
Let $M^n$ be an asymptotically flat manifold with a metric $g\in W^{1,p}_{-\tau}, p>n, \tau>\frac{n-2}{2}$, that is,
by Definition 2.1 of Bartnik in \cite{MR849427}, there exists a compact subset $K\subset M$ and a diffeomorphism $\Phi: M\backslash K\rightarrow\mathbb{R}^n\backslash\overline{B(0,1)}$,
such that $\Phi_*g\in W^{1,p}_{loc}(\mathbb{R}^n)$ and
\[ ||\Phi_*g(x)-\delta(x)||_{W^{1,p}_{-\tau}}<+\infty, \]
where the weighted Sobolev norm $W^{1,p}_{-\tau}$ is defined in Definition 1.1 of \cite{MR849427}.
By the Sobolev embedding of the weighted Sobolev space ((1.11) of \cite{MR849427}), we have,
for some $\alpha\in(0, 1-\frac{n}{p}]$,
$\Phi_*g\in W^{1,p}_{-\tau}\subset C^{0,\alpha}_{-\tau}$,
where the weighted Holder space is also defined in (1.12) of \cite{MR849427}, and
for any $x\in\mathbb{R}^n\backslash\overline{B(0,1)}$,
\[ |\Phi_*g-\delta|(x)=C|x|_{\delta}^{-\tau}. \]

If the scalar curvature distribution $\ll R_g, u \gg\geq\kappa$, for any compactly supported nonnegative smooth function $u$, where $\kappa\geq 0$ is a constant.
By the asymptotically flatness, $(M\backslash K, g)$ is isometric to $(\mathbb{R}^n\backslash\overline{B(0,1)}, \Phi_*g)$.
Then, for any nonnegative $\phi\in C^{\infty}_0(\mathbb{R}^n\backslash\overline{B(0,1)})$ and $u=\phi(\Phi)$, we have
\begin{equation*}
  \ll R_{\Phi_*g}, \phi \gg=\ll R_g, u \gg\geq\kappa.
\end{equation*}

For any $x\in M\backslash K$, denote
\[ |\Phi(x)|_{\delta}=|x_0|_{\delta}=r>>1. \]
Choose a neighborhood $\mathcal{U}_x$ of $x$, such that $\Phi(\mathcal{U}_x)=B_{\delta}(x_0, r_1)$, for some $0<r_1<<1$.
Then we take a cut-off function $\chi\in C^{\infty}_0(\mathbb{R}^n)$ satisfying $0\leq\chi\leq1$, $|\nabla\chi|\leq c_2$ and
\begin{equation*}
  \chi=\begin{cases}
         1, & \ \   \mbox{on \ \ } B_{\delta}(x_0, r_1), \\
         0, &  \ \   \mbox{on \ \ } \mathbb{R}^n\backslash B_{\delta}(x_0, r_1+1).
       \end{cases}
\end{equation*}

Let $g_0=\chi\Phi_*g+(1-\chi)\delta$, then $g_0\in W^{1, p}(\mathbb{R}^n)$ and
$||g_0-\delta||_{C^0(\mathbb{R}^n)}\leq ||\Phi_*g-\delta||_{C^0(\mathbb{R}^n)}=c_1r^{-\tau}=\epsilon$, for some $\epsilon<1$.
The scalar curvature distribution $\ll R_{g_0}, \phi \gg$ of $g_0$ can be defined by Definition \ref{scalar2}.
We will always choose the background metric $h=\delta$ in Definition \ref{scalar2}.
We have the following lemma about the scalar curvature distribution.

\begin{lemma}\label{g00}
For any smooth nonnegative function $\phi\in C^{\infty}_0(B_{\delta}(x_0, r_1+1))$,
the scalar curvature distribution of $g_0$ is
  \begin{equation}\label{g0}
    \ll R_{g_0}, \phi \gg=\ll R_{\Phi_*g}, \phi \gg +c_3\epsilon,
  \end{equation}
\end{lemma}
where $c_3$ depends on $\chi, \phi$.
\begin{proof}
  Since $g_0=\delta+\chi(\Phi_*g-\delta)$ and $\chi\leq1$, $|\Phi_*g-\delta|(x)=C_1r^{-\tau}<1$, we have
  \[ g_0^{-1}=[\delta+\chi(\Phi_*g-\delta)]^{-1}=\delta-\chi(\Phi_*g-\delta)+O(r^{-2\tau}), \]
  \[ (\Phi_*g)^{-1}=(\delta+\Phi_*g-\delta)^{-1}=\delta-(\Phi_*g-\delta)+O(r^{-2\tau}). \]
  So
  \[ g_0^{-1}=\delta-\chi(\delta-(\Phi_*g)^{-1})+O(r^{-2\tau})=\delta+\chi((\Phi_*g)^{-1}-\delta)+O(r^{-2\tau}). \]
  Then by definition, we can compute
  \begin{equation*}
    \begin{split}
       \Gamma^k_{ij}(g_0)= & \frac{1}{2}g_0^{kl}(\partial_ig_{0jl}+\partial_jg_{0il}-\partial_lg_{0ij}) \\
        = & \chi^2\Gamma^k_{ij}(\Phi_*g)+\chi(1-\chi)|\partial\Phi_*g|+\chi|\nabla\chi||\Phi_*g-\delta|
        +(1-\chi)|\nabla\chi||\Phi_*g-\delta|+O(r^{-3\tau})\\
        = & \chi^2\Gamma^k_{ij}(\Phi_*g)+I+O(r^{-3\tau}),
    \end{split}
  \end{equation*}
  where $supp I= supp\{\nabla\chi\}$ and $I=O(r^{-\tau})$.
  And
  \begin{equation*}
    V^k(g_0)=g_0^{ij}g_0^{kl}(\partial_jg_{0il}-\partial_lg_{0ij})=\chi^3V^k(\Phi_*g)+I+O(r^{-3\tau}),
  \end{equation*}
  \begin{equation*}
    \begin{split}
       F(g_0)= & -\partial_kg_0^{ij}\Gamma_{ij}^k(g_0)+\partial_kg_0^{ik}\Gamma_{ji}^j(g_0)+
       g_0^{ij}(\Gamma^k_{kl}(g_0)\Gamma^l_{ij}(g_0)-\Gamma^k_{jl}(g_0)\Gamma^l_{ik}(g_0)) \\
        = & -\chi^3\partial_k(\Phi_*g)^ij\Gamma_{ij}^k(\Phi_*g)+\chi^3\partial_k(\Phi_*g)^{ik}\Gamma^j_{ji}(\Phi_*g)\\
       & + \chi^5\Phi_*g^{ij}(\Gamma_{kl}^k(\Phi_*g)\Gamma_{ij}^l(\Phi_*g)-\Gamma^k_{jl}(\Phi_*g)\Gamma^l_{ik}(\Phi_*g))+
        I+O(r^{-2\tau}).
    \end{split}
  \end{equation*}
  Then we get
  \begin{equation*}
    \begin{split}
       \ll R_{g_0}, \phi \gg =& \int_{\mathbb{R}^n}\left( -V\cdot\partial\left(\phi\frac{d\mu_{g_0}}{d\mu_{\delta}}\right)+ F\phi\frac{d\mu_{g_0}}{d\mu_{\delta}} \right) d\mu_{\delta} \\
        = & \int_{B_{\delta}(x_0, r_1+1)}\left( -V\cdot\partial\left(\phi\frac{d\mu_{g_0}}{d\mu_{\delta}}\right)+ F\phi\frac{d\mu_{g_0}}{d\mu_{\delta}} \right) d\mu_{\delta} \\
        = & \ll R_{\Phi_*g}, \phi \gg +
         Cr^{-\tau}\int_{supp\{I\}}(|\partial\phi|+\phi) d\mu_{g_0}.
    \end{split}
  \end{equation*}
\end{proof}

By Theorem \ref{RDF}, there exists a smooth global solution $(g_t)_{t>0}$ of the Ricci-DeTurck flow with respect to the Euclidean background metric and $g(t)\rightarrow g_0$ in $C^0_{loc}$ as $t\rightarrow0$.
Actually, since $\Phi_*g\in W^{1,p}_{-\tau}$, we have better convergence.

\begin{lemma}\label{1n}
  For any $x\in\mathbb{R}^n$, there is an $\epsilon_1<1$ such that
  \[ \int_{B(x, 1)}|\nabla g_0|^nd\mu_{\delta}<\epsilon_1, \]
  where $\nabla=\partial$ is the connection of the Euclidean metric $\delta$.
\end{lemma}
\begin{proof}
Since $\nabla g_0$ is supported in $B_{\delta}(x_0, r_1+1)$, we only consider $x\in B_{\delta}(x_0, r_1+1)$. By $g_0=\delta+\chi(\Phi_*g-\delta)$, we have
\[ |\nabla g_0|=|\chi\nabla\Phi_*g+\nabla\chi(\Phi_*g-\delta)|\leq c_2\epsilon+|\nabla\Phi_*g|. \]
  By the Holder inequality,
  \begin{equation*}
  \begin{split}
     \int_{B(x,1)}|\nabla \Phi_*g|^nd\mu_{\delta}\leq &
     \left(\int_{B(x,1)}|\nabla \Phi_*g|^pr^{(\tau+1)p-n}d\mu_{\delta}\right)^{\frac{n}{p}}
     \left(\int_{B(x,1)}r^{\eta}d\mu_{\delta}\right)^{\frac{p-n}{p}} \\
      \leq & C||\nabla \Phi_*g||^n_{L^{p}_{-\tau-1}}r^{\eta(p-n)/p},
  \end{split}
  \end{equation*}
  where $r=|x|_{\delta}$, $\eta=(n-(\tau+1)p)\cdot\frac{n}{p-n}<0$.
\end{proof}

By Lemma \ref{1n} and Theorem 1.1 of Chu-Lee in \cite{arxiv2}, the smooth solution $(g_t)_{t>0}$ of the Ricci-DeTurck flow satisfies $g(t)\rightarrow g_0$ in $W^{1,n}_{loc}$ as $t\rightarrow0$ and $g(t)\rightarrow g_0$ in $C^{\infty}_{loc}(\Omega)$ if $g_0\in C^{\infty}_{loc}(\Omega)$ as $t\rightarrow\infty$.

By the definition of the scalar curvature distribution, we know, for any compactly supported smooth function $\phi$,
\[ \ll R(g_T), \phi \gg=\int_{\mathbb{R}^n}R(g_T)\phi d\mu_{g_T}, \]
where $T>0$.
We now prove Theorem \ref{dfr}.

We will need an estimate of the length along the Ricci flow due to Perelman \cite{arxiv3}.
\begin{lemma}(Lemma 8.3(a) in \cite{arxiv3})\label{Pel}
  Suppose g(t) is a Ricci flow and $Ric(x, t_0)\leq(n-1)K$ when $d_{t_0}(x, x_0)<r_0$.
  Then the distance function $d(x, t)=d_{t}(x, x_0)$ satisfies at $t=t_0$ outside $B_{t_0}(x_0, r_0)$ the differential inequality
  \[ (\partial_t-\Delta)d\geq -(n-1)(\frac{2}{3}Kr_0+r_0^{-1}), \]
  in the barrier sense.
\end{lemma}

\begin{thm}
    Let $M^n$ be an asymptotically flat manifold with a metric $g\in C^0\bigcap W^{1, p}_{-\tau}$, $p>n$, $\tau>\frac{n-2}{2}$. If the scalar curvature distribution $\ll R_g, u\gg\geq\kappa$, for any compactly supported nonnegative smooth function $u$ and some constant $\kappa\geq0$, then for some $\beta\in(0, \frac{1}{2})$, $g$ has scalar curvature bounded below by $\kappa$ in the sense of Ricci flow at infinity.
\end{thm}

\begin{proof}
    Let $C>0$ be any constant and  some $\beta\in(0, \frac{1}{2})$.
    Denote $\Phi(x)=x_0\in\mathbb{R}^n$ and $|x_0|_{\delta}=r$.
    Let $U_1=B_{\delta}(x_0, CT^{\beta})\subset\mathbb{R}^n$ be a subset with $0<T<T_0<1$. 
     Let $\Phi^{RD}(x,t;y,s)$ be the scalar heat kernel for the operator $\partial_t-\Delta_{g(t)}+\nabla_X^{g(t)}$  with Ricci-DeTurck flow background as in Theorem \ref{Rt0}.
     For any nonnegative smooth compactly supported function $\varphi\in C^{\infty}_0(U_1)$, consider the conjugate heat equation with Ricci-DeTurck flow background satisfying
\begin{equation}\label{hk}
  \begin{cases}
    \partial_t \varphi_t(x)= -\Delta_{g(t)}\varphi_t(x)+R(x,t)\varphi_t(x)-\nabla_X^{g(t)}\varphi_t(x), & \ \ \ \mathbb{R}^n\times(0, T],  \\
    \varphi_T(x)= \varphi(x), &\ \ \ \text{on} \ \mathbb{R}^n.
  \end{cases}
\end{equation}

Then $\varphi_t(x)$ can be expressed as
\begin{equation}\label{varphi}
  \varphi_t(x)=\int_{\mathbb{R}^n}\Phi^{RD}(y, T; x, t)\varphi(y)d\mu_T(y),
\end{equation}
where $d\mu_T=d\mu_{g(T)}$. 
We see that $\varphi_t$ is nonnegative but not compactly supported anymore.
Thus, we need a cutoff function.
We will use the approach of Huang-Lee in \cite{MR4517268}.
Let $\phi: [0, \infty)\rightarrow [0,1]$ be a smooth cutoff function with
\begin{equation*}
  \phi=\begin{cases}
    1, & \mbox{on } \ \ \ [0,\frac{1}{2}] \\
    0, & \mbox{on}\ \ \ \ [1, \infty),
  \end{cases}
\end{equation*}
and satisfying $|\phi'|\leq c_4$, $\phi'\leq0$, $\phi''\geq -c_4\phi$.
From the construction of $g_0$, we know $||g_0-\delta||_{C^0(\mathbb{R}^n)}<c_1r^{-\tau}=\epsilon<1$ and by \eqref{gt}\eqref{dgt}, we have, for $t>0$,
\begin{equation}\label{glip}
  ||g(t)-\delta||_{C^0(\mathbb{R}^n)}\leq c_2(n)\epsilon,
\end{equation}
\begin{equation}\label{Rmm}
 |Rm(g(t))|\leq \frac{c_3\epsilon}{t}.
\end{equation}

We define
\begin{equation*}
  \psi_t(x)=\phi\left(t^{\gamma}d_t(x, x_0)\right),
\end{equation*}
where $d_t(x, x_0)$ is the distance function to $x_0$ with respect to $g(t)$ and $\gamma>0$ is a constant.

Then, for $t\in(0, T_0]$, we have
\[ \psi_t(x)\equiv1 \]
on $B_{\delta}(x_0, \frac{1}{2(1+c_2\epsilon)t^{\gamma}})$. 
And, outside $B_{\delta}(x_0, \frac{1}{(1-c_2\epsilon)t^{\gamma}})$, we have $\psi_t\equiv0$.

Noting that $g(t)$ is the Ricci-DeTurck flow which differs from the Ricci flow by a diffeomorphism generated by $X$. More presicely, let $(\Phi_t)_{t>0}$ be a flow of time-dependent family of vector fields $X$, such that
\[ \partial_t\Phi_t(x)=X(\Phi_t(x)). \]
Then $\tilde{g}_t(x)=\Phi^*_tg_t(\Phi_t(x))$ satisfies the Ricci flow.
So $d_t(x)=\tilde{d}_t(\Phi_t^{-1}(x))$ and 
\[ \partial_t\tilde{d}_t(y)=\partial_td_t(x)+\nabla_X^{g(t)}d_t(x) \]
with $y=\Phi_t^{-1}(x)$ and $y_0=\Phi_t^{-1}(x_0)$.

By taking $r_0=\frac{1}{2t^{\gamma}}$, $t_0=t\in (0, T]$ in Lemma \ref{Pel}, from \eqref{Rmm}, there holds
\begin{equation}\label{cuto}
     (\partial_t-\Delta)\tilde{d}_t(y, y_0)\geq  -(n-1)\left( \frac{2c_3\epsilon}{3t}\frac{1}{2t^{\gamma}}+2t^{\gamma} \right) 
\end{equation}
in the barrier sense when $d_t(x, x_0)\geq \frac{1}{2t^{\gamma}}$.
Thus, we can calculate
\[ \partial_t\psi_t(x)=\phi'\left[t^{\gamma}(\partial_t\tilde{d}_t-\nabla_X^{g(t)}d_t)+\gamma t^{\gamma-1}d_t\right],
\]
\[ \Delta\psi_t(x)=\phi''t^{2\gamma}+\phi't^{\gamma}\Delta \tilde{d}_t. \]
Hence, by \eqref{cuto}, we have
\begin{equation}\label{cutev}
\begin{split}
  (\partial_t-\Delta)\psi_t= & \phi't^{\gamma}\left[\partial_t\tilde{d}_t-\Delta \tilde{d}_t+\frac{\gamma d_t}{t}\right]-\phi''t^{2\gamma}-\nabla_X^{g(t)}\psi_t(x) \\
    \leq & c_4t^{2\gamma}\psi_t+c_4t^{\gamma}(n-1)\left( \frac{2c_3\epsilon}{3t}\frac{1}{2t^{\gamma}}+2t^{\gamma} \right)-\nabla_X^{g(t)}\psi_t(x)\\
    \leq & c_4t^{2\gamma}\psi_t+\frac{c_5}{t}-\nabla_X^{g(t)}\psi_t(x),
\end{split}
\end{equation}
where $c_5$ depends on $n, c_3, c_4, \epsilon, \gamma$.

Denote $f(x, t)=(R-\kappa)_-$, then
\begin{equation}\label{fev}
  (\partial_t-\Delta)f=(\partial_t-\Delta)(-R)=-\nabla_X^{g(t)}f-2|Ric|^2.
\end{equation}

Consider the energy functional
\[ E(t)=\int_{\mathbb{R}^n}f\varphi_t\psi_td\mu_t. \]
Note that the integrand $f\varphi_t\psi_t$ is nonnegative smooth compactly supported for $t>0$.
Then, along the Ricci-DeTurck flow $g(t)$, by \eqref{hk}\eqref{cutev}\eqref{fev} and $\frac{d}{dt}d\mu_t=-(R+\text{div}X)d\mu_t$,  we have
\begin{equation}\label{enev}
  \begin{split}
    \frac{d}{dt}E(t)= & 
    \int_{\mathbb{R}^n}\left( \partial_t(f\varphi_t\psi_t)-(R+\text{div}X)f\varphi_t\psi_t \right)d\mu_t  \\
      = & \int_{\mathbb{R}^n}[(\Delta f-\nabla_X^{g(t)}f-2|Ric|^2)\varphi_t\psi_t+f(-\Delta\varphi_t+R\varphi_t-\nabla_X^{g(t)}\varphi_t)\psi_t
      +f\varphi_t(\partial_t-\Delta)\psi_t \\
      & +f\varphi_t\Delta\psi_t-(R+\text{div}X)f\varphi_t\psi_t ]d\mu_t\\
      \leq & \int_{\mathbb{R}^n}\left[ -2\varphi_t\nabla f\nabla\psi_t+f\varphi_t\left( c_4t^{2\gamma}\psi_t+\frac{c_5}{t} \right) \right]d\mu_t\\
      \leq & c_4t^{2\gamma}E(t)+2\int_{supp\{\phi'\}}\varphi_t|\nabla f||\nabla\psi_t|d\mu_t+\frac{c_5}{t}\int_{supp\{\phi'\}}f\varphi_td\mu_t.
    \end{split}
\end{equation}

By the definition of $\psi_t$ and \eqref{dgt}, we know
\[ |\nabla\psi_t|=|\phi't^{\gamma}\nabla^{g(t)} d_t|\leq c_4t^{\gamma}, \]
\begin{equation}\label{df}
  |\nabla f|=|\nabla R|\leq\frac{c_6(n)\epsilon}{t^{\frac{3}{2}}}.
\end{equation}
Along with \eqref{enev}, there holds
\begin{equation}\label{energy}
\begin{split}
    \frac{d}{dt}E(t)\leq&  c_4t^{2\gamma}E(t)+c_6\epsilon (t^{\gamma-\frac{3}{2}}+t^{-2})\int_{supp\{\phi'\}}\varphi_td\mu_t\\
    \leq&  c_4t^{2\gamma}E(t)+c_5\epsilon t^{-2}\int_{supp\{\phi'\}}\varphi_td\mu_t,
\end{split}     
\end{equation}
by changing $c_5$ accordingly.

By the definition of the cutoff function $\phi$, we see that
\[ supp\{\phi'\}\subset U_2=B_{\delta}(x_0, \frac{1}{(1-c_2\epsilon)t^{\gamma}})\backslash B_{\delta}(x_0,  \frac{1}{2(1+c_2\epsilon)t^{\gamma}}). \]

From \eqref{varphi}, we have
\begin{equation}\label{phitg}
  \begin{split}
    \int_{supp\{\phi'\}}\varphi_td\mu_t = & \int_{U_2}\int_{U_1}\Phi^{RD}(y, T; x, t)\varphi(y)d\mu_T(y)d\mu_t(x) \\
     \leq  & ||\varphi||_{L^{\infty}(U_1)}\int_{U_2}\int_{U_1}\Phi^{RD}(y, T; x, t)d\mu_T(y)d\mu_t(x).
  \end{split}
\end{equation}
To estimate the right hand side, we will use the Davies' double integral upper bound of the heat kernel with Ricci-DeTurck flow background given by Theorem 1.2 in \cite{MR3449222}.
In our case, using \eqref{glip} and \eqref{Rmm} in the proof of Theorem 1.2 of \cite{MR3449222}, we have the following result. 

\begin{thm}
    Let $\Phi^{RD}(y, t; x, l)$ be the heat kernel with Ricci-DeTurck flow background metric $g(t)$ obtained from $g_0$.
    Let $U_1, U_2$ be two compact sets of $\mathbb{R}^n$.
    Then, for $0<t<T<T_0$, we have
    \begin{equation}\label{hkd}
    \begin{split}
    \int_{U_2}\int_{U_1}\Phi^{RD}(y, T; x, t)d\mu_T(y)d\mu_t(x)
    \leq& \left(\frac{T}{t}\right)^{\frac{c_3\epsilon}{2}}
  \exp\left(-\frac{d_{\delta}^2(U_1, U_2)}{2(1+c_2\epsilon)^2(T-t)}\right)\\
  &\cdot\Vol_T(U_1)^{\frac{1}{2}}\Vol_t(U_2)^{\frac{1}{2}},
 \end{split}
 \end{equation}
where $d_{\delta}$ means the Eculidean distance.
\end{thm}

Noting that
\[ \Vol_T(U_1)\leq(1+c_2\epsilon)^{\frac{n}{2}}(CT^{\beta})^{n}, \]
\[ \Vol_t(U_2)\leq(1+c_2\epsilon)^{\frac{n}{2}}(2t^{-\gamma})^{n}, \]
\[ d_{\delta}(U_1, U_2)=((1-c_2\epsilon)t^{\gamma})^{-1}-CT^{\beta}>t^{-\gamma}, \]
inserting these into \eqref{phitg} arises
\begin{equation}\label{dtg}
  \int_{supp\{\phi'\}}\varphi_td\mu_t\leq  c_7T^{\frac{c_3\epsilon+\beta n}{2}}t^{\frac{-\gamma n-c_3\epsilon}{2}}\exp(-c_5t^{-2\gamma}T^{-1}),
\end{equation}
where $c_7$ depends on $\|\varphi\|_{L^{\infty}(U_1)}$ and $c_5$.
We obtain by \eqref{energy},
\begin{equation}\label{energy1}
\begin{split}
    \frac{d}{dt}E(t)
     \leq& c_4t^{2\gamma}E(t)+c_7\epsilon T^{\frac{c_3\epsilon+\beta n}{2}}t^{\frac{-\gamma n-c_3\epsilon}{2}-2}\exp(-c_5t^{-2\gamma}T^{-1})\\
    \leq& c_4t^{2\gamma}E(t)+c_7\epsilon T^{\frac{c_3\epsilon+\beta n}{2}}t^{\frac{-\gamma n-c_3\epsilon}{2}-2}(c_5t^{-2\gamma}T^{-1})^{-\eta}\\
    \leq& c_4t^{2\gamma}E(t)+c_7\epsilon T^{\frac{c_3\epsilon+\beta n}{2}+\eta}t^{\frac{-\gamma n-c_3\epsilon}{2}-2+2\gamma\eta},
\end{split}     
\end{equation}
for any $\eta>\frac{1}{2\gamma}(\frac{n\gamma+c_3\epsilon}{2}+1)>1$.

Integrate \eqref{energy1} over $[t, T]$, we obtain that
\begin{equation}\label{et}   
 E(T)\leq e^{\frac{c_4}{1+2\gamma}T^{1+2\gamma}}(E(t)+c_7\epsilon T^{\frac{\beta n-\gamma n}{2}+\eta+2\gamma\eta-1}). 
\end{equation}
Since $\ll R_g, u\gg\geq\kappa$, i.e. $\ll (R_g-\kappa)_-, u\gg=0$,
by Lemma \ref{g00}, we see that  
\[ \ll(R_{g_0}-\kappa)_-, u\gg\leq c_3\epsilon, \]
for any $u\in C_0^{\infty}(\mathbb{R}^n)$.
By construction of $g_0$, we know that $g_0=\delta$ outside $B_{\delta}(x_0, r_1+1)$. Hence, by Lemma \ref{1n}, we obtain that for any $u\in C^{\infty}(\mathbb{R}^n)$,
\[ \lim_{t\rightarrow0}\ll R(t), u\gg=\ll R_{g_0}, u\gg. \]

Therefore, we have
\begin{equation*}
    \lim_{t\rightarrow0} E(t)=\lim_{t\rightarrow0}\ll (R(t)-\kappa)_-, \varphi_t\psi_t\gg
    =\ll (R_{g_0}-\kappa)_-, u\gg\leq c_3\epsilon,
\end{equation*}
for $u=\lim_{t\rightarrow0}\varphi_t$. From \eqref{et}, we get
\[ E(T)\leq e^{\frac{c_4}{1+2\gamma}T^{1+2\gamma}}(c_3\epsilon+c_7\epsilon T^{\frac{\beta n-\gamma n}{2}+\eta+2\gamma\eta-1}). \]
Let $T=r^{-\lambda}$, $\lambda>0$, by $\epsilon=c_1r^{-\tau}$, there holds
\[ E(T)\leq c_7r^{-\tau}. \]
By definition of $E(t)$, we have for any $\varphi\in C_0^{\infty}(B_{\delta}(x_0, CT^{\beta}))$,
\[ \int_{\mathbb{R}^n}(R(T)-\kappa)_-\varphi(x)d\mu_T(x)\leq c_7r^{-\tau}. \]
Then for any $t=r^{-\lambda}$, in $B_{\delta}(x_0, Ct^{\beta})$, we have
\[ R(t)-\kappa\geq -(R(t)-\kappa)_-\geq -c_7r^{-\tau}, \]
which implies
\[ \inf_{C>0}\lim_{t\rightarrow0}\liminf_{B_{\delta}(x_0, Ct^{\beta})}R(t)\geq\kappa-\lim_{t\rightarrow0}c_7r^{-\tau}=\kappa. \]
So we proved $R_{C^0_{\beta}}(g_0)\geq\kappa$ at infinity.
\end{proof}

\begin{remark}
From the proof, we see that the scalar curvature distribution is in local case and the scalar curvature lower bound in the Ricci flow sense is pointwise. So we can only prove the two lower bounds of the scalar curvatue conincide at infinity.

    For $g\in W^{1,p}_{-\tau}$, we have better estimates of the curvature as in \cite{MR4596043}, but we do not need it in the proof. We state as the following.
\end{remark}

\begin{thm}\label{W1p}
  Let $g_0\in W^{1,p}_{-\tau}$, $p>n$ be a metric on $\mathbb{R}^n$.
  Suppose $\int_{B_{\delta}(x, 1)}|\partial g_0|^pd\mu_{\delta}\leq A$ for any $x\in\mathbb{R}^n$, then there exists a $T_0=T_0(n, A, p)$, such that along the Ricci-DeTurck flow $g(t)$ in Theorem \ref{RDF}, for $t\in(0, T_0]$, we have
  \begin{enumerate}
    \item $\int_{B_{\delta}(x, 1)}|\partial g(t)|^pd\mu_{\delta}\leq C(n, p)A$,
    \item $|\partial g(t)|\leq\frac{C(n, p)(A+1)}{t^{\frac{n}{2p}}}$,
    \item $|\partial^2 g(t)|\leq\frac{C(n, p)(A+1)}{t^{\frac{n}{4p}+\frac{3}{4}}}$,
  \end{enumerate}
  where $C(n, A, p)$ is a positive constant.
\end{thm}

Noting that from Theorem \ref{W1p} (3), we have the curvature bounded by $Ct^{-\alpha}$, where $\alpha=\frac{n}{4p}+\frac{3}{4}$. 
Comparing this with \eqref{Rmm}, we lost the $\epsilon$ in the constant.

\hspace{0.4cm}

\bibliography{mass}

\end{document}